\documentclass[11pt]{amsart}

\usepackage{amsmath,amsthm} 
\usepackage {latexsym}
\usepackage{amssymb}
\usepackage{xcolor}

\usepackage[utf8]{inputenc}
\newcommand{\beal}{\begin{align}}
\newcommand{\enal}{\end{align}}
\newcommand{\bealn}{\begin{align*}}
\newcommand{\enaln}{\end{align*}}
\newcommand{\bear}{\begin{eqnarray}}
\newcommand{\eear}{\end{eqnarray}}
\newcommand{\beeq}{\begin{equation}}
\newcommand{\eneq}{\end{equation}}

\newcommand{\eps}{{\varepsilon}}
\newcommand{\R}{{\mathbb R}}

\def\bm{\left[ \begin{array}{cc}}
\def\endm{\end{array}\right]}

\def\eps{\varepsilon}

\def\bm{\left[\begin{matrix} }
\def\endm{\end{matrix}\right]}

\def\R{{\mathbb R}}

\newtheorem{theorem}{Theorem}
\newtheorem{lemma}[theorem]{Lemma}
\newtheorem{defi}[theorem]{Definition}
\newtheorem{cor}[theorem]{Corollary}

\newtheorem{proposition}[theorem]{Proposition}

\newtheorem{remark}[theorem]{Remark}

\renewcommand{\epsilon}{\eps}
\renewcommand{\tilde}{\widetilde}
\numberwithin{equation}{section}
\numberwithin{theorem}{section}

\usepackage{xr-hyper} 
\usepackage[colorlinks]{hyperref}

\begin{document}
\title{Quantitative rapid and finite time stabilization of the heat equation}
\author{Shengquan Xiang}
\address{Bâtiment des Mathématiques, EPFL\\
 Station 8, CH-1015 Lausanne, Switzerland}
\email{\texttt{shengquan.xiang@epfl.ch.}}
\begin{abstract} 
The null controllability of the heat equation is known for decades \cite{Fatto-Russell, Fursikov-Imanuvilov-book-1997, Lebeau-Robbiano-CPDE}.  The finite time stabilizability of the one dimensional heat equation  was  proved   by Coron--Nguyên \cite{2017-Coron-Nguyen-ARMA},  while the same question for high dimensional spaces remained widely open. Inspired by  Coron--Trélat \cite{Coron-trelat-2004}   we find explicit stationary feedback laws that  quantitatively exponentially stabilize the heat equation  with decay rate $\lambda$ and $Ce^{C\sqrt{\lambda}}$ estimates, where Lebeau--Robbiano's spectral inequality \cite{Lebeau-Robbiano-CPDE} is naturally used.  Then a  piecewise controlling argument  leads to null controllability with optimal cost $Ce^{C/T}$, as well as  finite time stabilization.
\end{abstract}
\maketitle
\section{introduction}
Let $\Omega$ be an open domain in $\mathbb{R}^d$ with smooth boundary   and $\omega\subset \Omega$ an open subset.  We are interested in the stabilization and controllability of the  internal controlled heat equation, 
\begin{gather}
y_t= \Delta y+ 1_{\omega} f\;   \textrm{ in } \Omega,  \label{cons1}\\
y=0\; \textrm{ on } \partial \Omega.  \label{cons2}
\end{gather} 

\subsection{Stabilization problems}
It is well known  that  in the 90's  the null controllability of the above system was  simultaneously discovered by  Lebeau--Robbiano and  Fursikov--Imanuvilov via different approaches \cite{Fursikov-Imanuvilov-book-1997, Lebeau-Robbiano-CPDE},  rely on Lions' H.U.M.  \cite{Lions},  Russell's idea on  controllability and observability \cite{Russell-1978}, and most importantly Carleman estimates \cite{Carleman, Hormander-firstbook}.   See \cite{LR-Lebeau} for a complete and pedagogical introduction on these different but somehow complementary methods. Later on many people have contributed in the related controllability problems  \cite{Coron-trelat-2004, Sylvain-Zuazua-2011, FGIP-2004,  F-Zuazua-2, F-Zuazua-1}.

Though the study on the controllability of the heat equation  is nearly complete,  less is known concerning stabilization.  Generally speaking,  exponential stabilization for the heat equation  even for evolution  equations with  operators generating analytic semi-groups should be easier than null controllability problems. Indeed from a spectrum point of view  one needs to stabilize finitely many unstable modes, which  in some sense  is  easier than unique continuation problems, while null controllability corresponds to observability inequality.  The controlled wave equation is probably the best example to see this difference, Hörmander--Tataru--Robbiano--Zuily \cite{Hormander-97, L-L-JEMS, Robbiano-Zuily, Tataru-95} proved the  unique continuation  for arbitrary control domain, while observability requires the control domain satisfying G.C.C. \cite{Bardo-Lebeau-Rauch, Burq-Gerard-wave} according to Bardo--Lebeau--Rauch.   Several methods have been introduced for exponential stabilization problems  on partial differential equations, among which the most commonly used  should be the so-called Riccati method that  comes from finite dimensional optimal control theory (see for example \cite{Barbu-MAMS, Barbu-Triggiani, Kunisch-2014, Lions-opt, Lions, Raymond-NS}).  Modulo some  systematic arguments, in the end of the day it suffices to solve some non-linear algebraic Riccati equation in order to define a stabilizing feedback law.  Though powerful the feedback law is not  explicit, because it is an implicit solution of the Riccati equation, for which even the numerical computation is demanding. Thus it is always asked to design   simper and more efficient exponentially stabilizing feedback laws that  provide quantitative stabilizing estimates.

Finite time stabilization can be regarded as one of the ultimate questions to be asked for control theory, which  is definitely much more involved than  null controllability problems. In fact, even  for the one dimensional heat equation the finite time stabilization problem  was solved quite recently  by Coron--Nguyên \cite{2017-Coron-Nguyen-ARMA}, the controllability of which was known for nearly half century \cite{Fatto-Russell}.  We refer to the  paper  by Coron and the author   \cite[Introduction]{Coron-Xiang-2018} for a  detailed review on this problem. The crucial  point for \cite{2017-Coron-Nguyen-ARMA} is the exponential stabilization by stationary feedback laws for decay rate $\lambda$ with quantitative estimates $e^{C\sqrt{\lambda}}$ via backstepping method. The backstepping method, first introduced by Krstic and his collaborators \cite{2008-Krstic-Smyshlyaev-book}, corresponds to moving the spectrum with the help of some feedback laws. It has been improved  in  \cite{coron-2015, coronluqi} so that  can  be adapted to more one dimensional models \cite{Coron-Gagnon, ZhangRapidStab, Zhang-finite-2019}.    From a spectrum point of view, this method is different from any other stabilizing techniques concentrating on  finite dimensional  low frequency terms, as a result it can be applied to hyperbolic systems.  However, it is a challenging problem to introduce backstepping method for models in high dimensional spaces.  
Because the other stabilization methods are rather abstract  and the backstepping method provides satisfying $e^{C\sqrt{\lambda}}$ estimates, it was believed that the generalization of backstepping should appear before the proof of  finite time stabilization of the heat equation.

\subsection{The main results}
In this paper, we solve the finite time stabilization problem of the heat equation in any dimensional space, and provide a quantitative exponential stabilization method. Instead of using Riccati methods, or of generalizing backstepping to high dimensional spaces, we use a  straightforward Lyapunov functional method. It turns out that the exponential stabilization of the heat equation with arbitrary decay rate $\lambda$ can be achieved via simple and explicit feedback laws. Surprisingly, the spectral estimates found by Lebeau--Robbiano is naturally and elegantly used to provide an $e^{C\sqrt{\lambda}}$ stabilizing estimate.  Thanks to this powerful estimate, by applying a standard piecewise controlling argument we further prove  the null controllability via stabilization approach, and more importantly, solve the finite time stabilization problem with arbitrarily small time $T>0$ (hence small-time stabilization).\\

Before stating the detailed theorems, we briefly explain some terminologies used for finite time stabilization.
A \textit{time-varying feedback law} $U$ is an application 
\begin{equation*}
\left\{\begin{array}{ccc}
U:  \mathbb{R}\times L^2(\Omega) &\to &  L^2(\Omega)
\\
(t; y)&\mapsto & U(t; y).
\end{array}
\right.
\end{equation*}
A \textit{stationary} feedback law  is such an application only depends on $L^2(\Omega)$, and a $T$-\textit{periodic} feedback law is a time-varying feedback law  such that $U(t+T; y)= U(t; y)$. The \textit{closed-loop system} associated to  a feedback law $U$ is the evolution equation
\begin{gather}\label{clo-def-1}
\begin{cases}
y_t= \Delta y+ 1_{\omega} U(t; y), \; (t, x)\in (s, +\infty) \times \Omega,\\
y(t, x)= 0, \;  (t, x)\in (s, +\infty) \times \partial \Omega.
\end{cases}
\end{gather} 
Eventually  we are interested in $T$-\textit{periodic proper} feedback laws. Heuristically speaking, a feedback law $U$ is called \emph{proper} if the Cauchy problem associated to the closed-loop system
\eqref{clo-def-1} admits a unique solution for every $s\in \R$ and for every initial data $y_0 \in L^2(\Omega)$ at time $s$.  Therefore, formally we are allowed to define a ``\textit{flow}", $\Phi(t, s; y_0)$, as the state  at time  $t$ of the solution of \eqref{clo-def-1} with initial state $y(s, x)= y_0(x)$, where $y_0\in L^2(\Omega)$ and $t\geq s$.  Please follow Section \ref{sec-def-prop-feed} for precise  definitions on solutions of closed-loop systems, proper feedback laws,  ``flow"  with respect to systems with proper feedback laws, as well as finite time stabilization.\\

Successively we are able to prove the following theorems concerning rapid stabilization, null controllability, and finite time stabilization in Section \ref{sec-rap}, Section \ref{sec-null}, and Section \ref{sec-finite} respectively.
\begin{theorem}[Quantitative rapid stabilization]\label{int-thm-rap-sta-li}
There exists an effectively computable constant $C>0$ such that for any $\lambda> 0$ we construct an explicit stationary feedback law $\mathcal{G}_{\lambda}: L^2(\Omega)\rightarrow L^2(\Omega)$, such that  the closed-loop system 
\begin{gather*}
y_t= \Delta y+ 1_{\omega} \mathcal{G}_{\lambda} y \; \textrm{ in } \Omega,\\
y=0 \; \textrm{ on } \partial \Omega,
\end{gather*} 
 is exponentially stable:
\begin{align}
||\Phi(t, s; y_0)||_{L^2(\Omega)}+ ||1_{\omega}\mathcal{G}_{\lambda}\Phi(t, s; y_0)||_{L^2(\omega)}&\leq C e^{C\sqrt{\lambda}} e^{-\frac{\lambda}{2} (t-s)}||y_0||_{L^2(\Omega)}, \; \forall \; s\in \R,  \forall  \;t\in [s, +\infty). \notag
\end{align}
\end{theorem}

\begin{theorem}[Null controllability with optimal cost]\label{int-thm-null-col}
There exists an effectively computable constant $C>0$ such that, for any $T\in (0, 1)$, for any $y_0\in L^2(\Omega)$, we  find an explicit control $f|_{[0, T]}(t, x)$ for the control system \eqref{cons1}--\eqref{cons2} such that the unique solution verifies
\begin{equation}
y(0, x)= y_0(x) \textrm{ and } y(T, x)=0,  \notag
\end{equation}
moreover,  
\begin{equation*}
|| 1_{\omega} f(t, x)||_{L^{\infty}(0, T; L^2(\Omega))}\leq   e^{C/T} ||y_0||_{L^2(\Omega)}.
\end{equation*} 
\end{theorem}

\begin{theorem}[Semi-global finite time stabilization with explicit feedback laws]\label{int-thm-semi-stab}
For any $\Lambda \geq 1$, for any $T>0$, we construct an explicit  $T$-periodic proper feedback law $U$ satisfying 
\begin{equation*}
||1_{\omega} U(t; y)||_{L^2(\Omega)}\leq C||y||_{L^2(\Omega)}+ 2||y||_{L^2(\Omega)}^{1/2}, \; \forall  \;y\in L^2(\Omega), \; \forall \; t\in \R,
\end{equation*}
with some $C$ effectively computable, that stabilizes system \eqref{clo-def-1} in finite time:
\begin{itemize}
\item[(i)] ($2T$ stabilization) $\Phi(2T+ t, t; y_0)=0, \;\;\forall \;t\in \mathbb{R},\; \forall\; ||y_0||_{L^2(\Omega)}\leq \Lambda.$
\item[(ii)] (Uniform stability)
For every  $\delta> 0$ there exists an effectively computable $\eta> 0$ such that
\begin{equation*}
\left(||(y_0||_{L^2(\Omega)}\leq \eta\right) \Rightarrow \left(||\Phi(t, t'; y_0)||_{L^2(\Omega)}\leq \delta, \;\forall \;  t'\in \R,\; \forall\; t\in ( t', +\infty) \right).
\end{equation*}
\end{itemize}
\end{theorem}

\begin{remark}\label{rmk-why-unif}
Let us emphasize  that the ``uniform stability" condition   is one of  the essential differences between null controllability and finite time stabilization.    Indeed, this condition is crucial for stabilization problems as in reality  systems may have errors and exist perturbations, thus  the stabilizing system are required to overcome these difficulties.   Another main difficulty for closed-loop stabilization compared to open-loop control is that the feedback only depends on  current states, while control may depend on backward states.
\end{remark}

\noindent\textbf{Statement on notations:} for readers convenience we summarize some notations and  constants that will be defined  and used later on.  Moreover, once a constant is defined,  from then on we will use it directly.   Notations $(\tau_i, e_i)$ and $N(\lambda)$ about eigenvalues defined in Section \ref{sec-spec}; orthogonal projection $P_N, P_N^T$ defined after equation \eqref{def-Vy}; truncated operator $\mathcal{K}_r$  in \eqref{def-FF};  $\gamma_{\lambda}$ and $\mu_{\lambda}$ in \eqref{def-gam-mu};  feedback law $\mathcal{F}_{\lambda}$  in \eqref{def-Flambda}; $r_{\lambda}$ in \eqref{def-C2}; the partition $T_n, \lambda_n$, and $I_n$ by \eqref{def-In-t-l}.  All the following constants are independent of $\lambda> 0$:
 $C_1$ defined in Proposition \ref{prop-1}; $C_2$ in equation \eqref{def-C2}; $\Gamma$ by \eqref{def-Gam} and $C_3$ by  \eqref{es-123}.  \\

\noindent\textbf{Acknowledgments.}  The author would like to thank Jean-Michel Coron for having attracted his attention to this problem and for fruitful discussions. He also thanks Emmanuel Trélat, Klaus Widmayer, and Joachim Krieger for  valuable discussions on this problem.

\section{Rapid stabilization}\label{sec-rap}
\subsection{Well-posedness results} 
In this section we quickly review the well-posedness results for  the following   Cauchy problem
\begin{equation}\label{Cauch-heat}
\begin{cases}
y_t= \Delta y+  f(t, x), \;  (t, x)\in (t_1, t_2)\times \Omega,\\
y(t, x)=0, \;\; (t, x)\in (t_1, t_2)\times \partial \Omega,\\
y(t_1, x)= y_0(x),
\end{cases}
\end{equation} 
as well as the related closed-loop systems with stationary feedback laws $i. e. $ $f(t, x)= \mathcal{L} y$, where $\mathcal{L}$ is a bounded operator on $L^2(\Omega)$.

The well-posedness  for both open-loop systems and closed-loop systems with stationary feedback laws are well-known, here we adapt the definition of the solution in the transposition sense, for which  the well-poseness results are derived from classical Hille--Yosida semi-group theory.  Transposition sense solution is introduced   by Lions \cite{Lions}, for those who are not familiar with  those definitions,  we refer to the  book by Coron \cite[Chapter 1--2] {coron} for an excellent introduction on this subject.
\begin{defi}\label{def-sol-open}
Let $t_1, t_2\in \R$ be such that $t_1<t_2$. Let $y_0\in L^2(\Omega)$ and $ f(t, x)\in L^2(t_1, t_2;L^2(\Omega))$. A solution to the Cauchy problem \eqref{Cauch-heat} is a function $y\in C^0([t_1, t_2]; L^2(\Omega))\cap L^2(t_1, t_2; H^1_0(\Omega))$ such that, for every $\tau\in [t_1, t_2]$ and for every $\phi\in C^0([t_1, \tau]; H^1(\Omega))$ such that
\begin{gather*}
\phi_t\in L^2((t_1, t_2)\times \Omega), \; \Delta \phi\in L^2((0, T)\times \Omega), \textrm{ and } \phi(t, \cdot)\in H^1_0(\Omega) \; \forall \; t\in [t_1, \tau], 
\end{gather*}
one has
\begin{equation*}
\int_{\Omega} y(\tau, x) \phi(\tau, x)dx-\int_{\Omega} y_0(x) \phi(t_1, x) dx
- \int_{t_1}^{\tau}\int_{\Omega} f \phi dxdt
- \int_{t_1}^{\tau} \int_{\Omega} (\phi_t+\Delta \phi)y dxdt= 0.
\end{equation*}

\end{defi}
\begin{theorem}\label{thm-inho-es}
For any $T\in (0, 1]$, for any $y_0\in L^2(\Omega)$, and for any $f\in L^2(0, T; L^2(\Omega))$, the Cauchy problem \eqref{Cauch-heat} has a unique solution.  Moreover,  this solution satisfies
\begin{align*}
||y(t)||_{L^2(\Omega)}&\leq  ||y_0||_{L^2(\Omega)}+  2||f||_{L^2(0, t;  L^2(\Omega))}, \; \forall \; t\in (0, T],\\
 ||\nabla y||_{L^2(0, t; L^2(\Omega))}&\leq ||y_0||_{L^2(\Omega)}+  2||f||_{L^2(0, t; L^2(\Omega))}, \; \forall \; t\in (0, T].
\end{align*}
\end{theorem}

We do not recall the solution  definition to closed-loop systems with stationary feedback laws as classical. Besides it  can be covered by the more general definition of solutions for time-varying feedback systems that will be presented  in Section \ref{sec-def-prop-feed}.  Concerning closed-loop systems with stationary feedback laws  we have the following well-posedness results.
\begin{theorem}\label{thm-clo-sta}
Let  $\varphi_i\in L^2(\Omega), 1\leq i\leq n$ be given functions.  Let  $l_i: L^2(\Omega)\rightarrow \mathbb{R}, 1\leq i \leq n$ be given bounded linear  operators.  For any  $y_0\in L^2(\Omega)$  the Cauchy problem
\begin{equation*}\label{Cauch-heat-loop}
\begin{cases}
y_t= \Delta y+ 1_{\omega} \Big(\sum_{i=1}^n l_i(y) \varphi_i \Big), \;  (t, x)\in (0, T)\times \Omega,\\
y(t, x)=0, \; (t, x)\in (0, T)\times \partial \Omega,\\
y(0, x)= y_0(x),
\end{cases}
\end{equation*} 
has a unique solution.  
\end{theorem}
Similar results exist for non-linear Lipschitz stationary feedback laws,  the proof of which is a simple modification  based on fixed point arguments and \textit{a priori} estimates.
For $r\in (0, 1/2]$ we introduce the cutoff function  $ f_r\in C^{\infty}(\mathbb{R}) $ and  the operator $\mathcal{K}_r: L^2(\Omega)\rightarrow L^2(\Omega)$ satisfying
\begin{gather}
f_r(x)=1  \textrm{ for } x\in [0, r], \;f_r(x)=0 \textrm{ for } x\in [2r, +\infty), \;\textrm{ and } f_r(x)\in [0, 1] \textrm{ for } x\in \mathbb{R}, \notag\\
\mathcal{K}_r (y)= y \cdot f_r\left(||y||_{L^2(\Omega)}\right), \;\forall \; y\in L^2(\Omega).  \label{def-FF}
\end{gather}
\begin{theorem}\label{thm-clo-sta-F}
Let $T\in (0, 1]$.  Let $r\in (0, 1/2]$. Let  $\varphi_i\in L^2(\Omega), 1\leq i\leq n$ be given functions.  Let  $l_i: L^2(\Omega)\rightarrow \mathbb{R}, 1\leq i \leq n$ be given bounded linear operators.  For any  $y_0\in L^2(\Omega)$  the Cauchy problem
\begin{equation*}\label{Cauch-heat-loop-nonli}
\begin{cases}
y_t= \Delta y+ 1_{\omega} \mathcal{K}_r\Big(\sum_{i=1}^n l_i(y) \varphi_i \Big), \;  (t, x)\in (0, T)\times \Omega,\\
y(t, x)=0, \; (t, x)\in (0, T)\times \partial \Omega,\\
y(0, x)= y_0(x),
\end{cases}
\end{equation*} 
has a unique solution.  
\end{theorem}

\subsection{Spectral estimates}\label{sec-spec}
Let us consider the Laplace operator with Dirichlet boundary condition $\Delta: H^2(\Omega)\cap H^1_0(\Omega)\rightarrow L^2(\Omega)$, there is a Hilbert orthogonal basis of $L^{2}(\Omega)$:
\begin{gather*}
0<\tau_1\leq \tau_2\leq \tau_3\leq...\leq \tau_n\leq...,\\
-\Delta e_i=\tau_i e_i \textrm{ with } e_i|_{\partial \Omega}=0.
\end{gather*}
Different $\tau_n$ may coincident, but every eigenvalue only have finite algebraic multiplicity.   For any given positive number $\lambda>0$, we define $N(\lambda)$  the number of eigenvalues (counting multiplicity) that are not strictly bigger than $\lambda$, $i. e. $ $ \tau_{N(\lambda)}\leq \lambda<\tau_{N(\lambda)+1}$.   
Moreover, the distribution of $\{\tau_k\}_{k=1}^{\infty}$ obeys Weyl's law:  $N(\lambda)\sim (2\pi)^{-d} \omega_d\;\textrm{vol}(\Omega) \lambda^{d/2}$, where $\omega_d$ is the volume of the  unit ball.  For ease of notations, in the following, if there is no confusion sometimes we simply denote $N_{\lambda}$ by $N$.

\begin{proposition}\label{prop-1}
The eigenfunctions $\{e_i\}_{i=1}^{\infty}$ satisfy
\begin{itemize}
\item[1)] Orthonormal basis: $(e_i, e_j)_{L^2(\Omega)}= \delta_{ij}$.
\item[2)] (Unique continuation) The symmetric matrix $J_N$ given below is invertible,
\begin{equation}\label{def-JN}
J_N:= \left( (e_i, e_j)_{L^2(\omega)} \right)_{i, j=1}^N.   
\end{equation}
\item[3)] (Tunneling estimate) There exist $C_0>0$ that is independent of $\tau_n$ such that 
\begin{equation}
||e_n||_{L^2(\omega)}^2\geq C_0^{-1} e^{-C_0\sqrt{\tau_n}}.   \notag
\end{equation}
\item[4)] (Spectral estimate)  There exist $C_1\geq 1$ that is independent of $\lambda> 0$ such that 
\begin{equation}
||\sum_{i=1}^{N(\lambda)} a_i e_i||_{L^2(\omega)}^2\geq C_1^{-1} e^{-C_1\sqrt{\lambda}} \sum_{i=1}^{N(\lambda)} a_i^2.   \notag
 \end{equation}

\end{itemize}

\end{proposition}

\begin{proof}
1) This is a well-known result upon self-adjoint operators with compact resolvent. \\
2) This is a consequence of the unique continuation of the Dirichlet operator.  One can see Barbu--Triggiani \cite{Barbu-Triggiani} for 
 more general results.  \\
3)  First proved by Donnelly--Fefferman in \cite{D-Fefferman} for compact Riemannian manifolds, the latest related result  is given by Léautaud--Laurent \cite{L-L-MAMS} for hypoelliptic equations. \\
4) This highly non-trivial observation is found by Lebeau--Robbiano \cite{Lebeau-Robbiano-CPDE}  via Carleman estimates, which is essentially the core of  their proof on null controllability of the heat equation. Indeed the form $e^{\sqrt{\lambda}}$ is  optimal once $\overline{\omega}\neq \Omega$, as illustrated in \cite{LR-Lebeau}. However, the optimality of the constant $C_1$, which clearly depends on the geometry of $(\Omega, \omega)$, is still open.
\end{proof}
As a direct consequence of property $4)$ of the preceding proposition, we have a quantitative estimate of $J_N$ as quadratic form.
\begin{lemma}\label{Key-Lemma} 
For $Y_{N(\lambda)}=(a_1, a_2, ..., a_{N(\lambda)})$, we have
\begin{equation*}
Y_{N(\lambda)}^{T}J_{N(\lambda)} Y_{N(\lambda)}\geq C_1^{-1} e^{-C_1\sqrt{\lambda}} ||Y_{N(\lambda)}||_2^2.
\end{equation*}
\end{lemma}
Actually, letting $N$ be presenting $N_{\lambda}$,  we get
\begin{equation*}
Y_N^{T}J_N Y_N= \sum_{1\leq i, j\leq N} a_i \left(e_i, e_j\right)_{L^2(\omega)} a_j= \left(\sum_{i=1}^N a_i e_i, \sum_{j=1}^N a_j e_j\right)_{L^2(\omega)}= ||\sum_{i=1}^N a_i e_i||_{L^2(\omega)}^2\geq C_1^{-1} e^{-C_1\sqrt{\lambda}}||Y_N||_2^2.
\end{equation*}

\subsection{Rapid stabilization via Lyapunov function approach and explicit feedback law}
The following rapid stabilization result  is inspired by the Lyapunov function idea introduced by Coron--Trélat \cite{Coron-trelat-2004}, where it was used as an intermediate step for their proof of  global controllability of steady states of non-linear parabolic equations in one dimensional space.  This idea has been adapted to various models, for example \cite{Coron-Trelat-2006} on  global controllability of one dimensional wave equations and \cite{Trelat-stab-book} for others. However, though relatively efficient and effectively calculable,  no attempt on  quantitative estimates has been made.   Probably this is because in the proof some general theories as  Kalman's rank condition  and  stabilization matrix are used.

Instead of using  abstract stabilizing matrix arguments, here we construct precise Lyapunov functionals and quite surprisingly the spectral estimates by Lebeau--Robbiano are naturally  used. That is the reason we get a quantitative rapid stabilization result with $Ce^{C\sqrt{\lambda}}$ estimates.

For any given $\lambda>0$, we suggest control terms in forms of  $\sum_i^{N(\lambda)} e_i|_{\omega} u_i(t) $ with $u_i(t)\in \R$, thus consider the following controlled problem:
\begin{gather}
y_t= \Delta y+ 1_{\omega} \left( \sum_{i=1}^{N(\lambda)} e_i u_i(t) \right) \; \textrm{ in } \Omega, \label{clo-1} \\
y=0 \; \textrm{ on } \partial \Omega.
\end{gather} 
In the rest part of this section, we simply denote $N_{\lambda}$ by $N$.  By decomposing 
\begin{gather}
y(t, x)= \sum_{i=1}^{\infty} y_i(t) e_i,\;\;
1_{\omega} e_j=\sum_{i=1}^{\infty} (1_{\omega} e_j, e_i)_{L^2(\Omega)} e_i=  \sum_{i=1}^{\infty} (e_i, e_j)_{L^2(\omega)} e_i,
\end{gather}
and by defining 
\begin{equation}
X_N(t):= \begin{pmatrix}
y_1(t)\\ y_2(t)\\...\\y_N(t)
\end{pmatrix}, \;\;
U_N(t):= \begin{pmatrix}
u_1(t)\\ u_2(t)\\...\\u_N(t)
\end{pmatrix}, \;\;
A_N:= \begin{pmatrix}
-\tau_1 & & &\\ & -\tau_2 & & \\&&...&\\& & &-\tau_N
\end{pmatrix},
\end{equation}
we know, thanks to the definition of $J_N$ in  \eqref{def-JN}, that the finite dimensional system  $X_N(t)$ satisfies 
\begin{equation}
\dot{X}_N(t)= A_N X_N(t)+ J_N U_N.
\end{equation}

For any given $\lambda$  (thus $N$ is given),  for $\gamma_{\lambda}, \mu_{\lambda}>0$ that will be fixed later on, we suggest  the feedback law  
\begin{equation}\label{cls-2}
U_N(y(t)):= -\gamma_{\lambda} X_N(t),
\end{equation}
as well as the Lyapunov function:
\begin{equation}\label{def-Vy}
V(y):= \mu_{\lambda} ||X_N||_2^2+ \left(P_N^{\perp} y, P_N^{\perp} y\right)_{L^2(\Omega)}, \; \forall y\in L^2(\Omega),
\end{equation}
where  $||X_N||_2^2$ is given by $\sum_{i=1}^N y_i^2$,  $P_N$ is the projection on the sub-space spanned by $\{e_i\}_{i=1}^N$, and $P_N^{\perp}$ be  its co-projection.  Thanks to Theorem \ref{thm-clo-sta}, the closed-loop system \eqref{clo-1}--\eqref{cls-2} is well-posed.

According to the preceding feedback law, $y(t)$ and $X_N$ verify
\begin{gather}
\dot{X}_N(t)= A_N X_N(t)-\gamma_{\lambda} J_N X_N,\\
y_t= \Delta y- \gamma_{\lambda} 1_{\omega} \left( \sum_{i=1}^N e_i X_i(t) \right) \; \textrm{ in } \Omega,\\
y=0 \; \textrm{ on } \partial \Omega.
\end{gather}
Thus,
\begin{align*}
\frac{d}{dt}V\left(y(t)\right)&= \mu_{\lambda} \frac{d}{dt} ||X_N||_2^2+ \frac{d}{dt} \left(P_N^{\perp} y, P_N^{\perp} y\right)_{L^2(\Omega)}= \mu_{\lambda} \frac{d}{dt} ||X_N||_2^2+  2\left(P_N^{\perp} y,  \frac{d}{dt}y\right)_{L^2(\Omega)}.
\end{align*}
On the one hand we know that 
\begin{equation}
\mu_{\lambda} \frac{d}{dt} ||X_N||_2^2= \mu_{\lambda} X_N^T \left( A_N^T+ A_N-2\gamma_{\lambda} J_N\right)X_N= 2\mu_{\lambda} X_N^T \big( A_N- \gamma_{\lambda} J_N\big)X_N\leq -2\mu_{\lambda} \gamma_{\lambda}  C_1^{-1} e^{-C_1\sqrt{\lambda}}||X_N||_2^2.  \notag
\end{equation}
On the other hand we have 
\begin{align*}
\frac{d}{dt} \left(P_N^{\perp} y, P_N^{\perp}y\right)_{L^2(\Omega)}&= 2 \left(P_N^{\perp} y, y_t\right)_{L^2(\Omega)},\\
&= 2\left(P_N^{\perp} y, \Delta y- \gamma_{\lambda} 1_{\omega} \left( \sum_{i=1}^N e_i X_i(t) \right)\right)_{L^2(\Omega)}, \\
&= -2\sum_{i=N+1}^{\infty}\tau_i y_i^2-2\gamma_{\lambda}\left(P_N^{\perp} y,  1_{\omega} \left( \sum_{i=1}^N e_i X_i(t) \right)\right)_{L^2(\Omega)}, \\
&\leq -2\lambda ||P_N^{\perp} y||_{L^2(\Omega)}^2+2\gamma_{\lambda} ||P_N^{\perp} y||_{L^2(\Omega)}  ||X_N||_2, \\
&\leq -2\lambda ||P_N^{\perp} y||_{L^2(\Omega)}^2+ \lambda ||P_N^{\perp} y||_{L^2(\Omega)}^2+ \frac{\gamma_{\lambda}^2}{\lambda} ||X_N||_2^2,\\
&= -\lambda ||P_N^{\perp} y||_{L^2(\Omega)}^2+ \frac{\gamma_{\lambda}^2}{\lambda} ||X_N||_2^2.
\end{align*}
Thus
\begin{equation}
\frac{d}{dt}V(y(t))\leq -2\mu_{\lambda}\gamma_{\lambda}  C_1^{-1} e^{-C_1\sqrt{\lambda}}||X_N||_2^2-\lambda ||P_N^{\perp} y||_{L^2(\Omega)}^2+ \frac{\gamma_{\lambda}^2}{\lambda} ||X_N||_2^2.  \notag
\end{equation}
Motivated from the above estimate, we choose 
\begin{equation}\label{def-gam-mu}
\gamma_{\lambda}:= C_1e^{C_1\sqrt{\lambda}} \lambda, \;\;    \mu_{\lambda}:= \frac{\gamma_{\lambda}^2}{\lambda^2}= C_1^2e^{2C_1\sqrt{\lambda}},
\end{equation}
which further yields
\begin{align*}
\frac{d}{dt}V(y(t))&\leq- 2\mu_{\lambda} \lambda ||X_N||_2^2-\lambda ||P_N^{\perp} y||_{L^2(\Omega)}^2+ \mu_{\lambda} \lambda ||X_N||_2^2,\\
&\leq -\mu_{\lambda} \lambda ||X_N||_2^2- \lambda ||P_N^{\perp} y||_{L^2(\Omega)}^2,\\
&= -\lambda \left(\mu_{\lambda} ||X_N||_2^2+ ||P_N^{\perp} y||_{L^2(\Omega)}^2   \right)= -\lambda  V(y(t)).
\end{align*}
Since   $\mu_{\lambda}\geq 1$ for $C_1\geq 1$, we know that, 
\begin{equation}
||y(t)||_{L^2(\Omega)}^2\leq V(y(t))\leq e^{-\lambda t} V(y(0))\leq e^{-\lambda t}\mu_{\lambda} ||y(0)||_{L^2(\Omega)}^2\leq C_1^2 e^{2C_1\sqrt{\lambda}} e^{-\lambda t}||y(0)||_{L^2(\Omega)}^2,  \notag
\end{equation}
thus
\begin{equation}
||y(t)||_{L^2(\Omega)}\leq C_1 e^{C_1\sqrt{\lambda}} e^{-\frac{\lambda}{2} t}||y(0)||_{L^2(\Omega)}.  \notag
\end{equation}
Moreover since the control (feedback) is given by 
\begin{equation}
1_{\omega} f(t, x)= - \gamma_{\lambda} 1_{\omega} \left( \sum_{i=1}^N e_i X_i(t) \right), \notag
\end{equation}
we know that
\begin{equation}
|| f(t, \cdot)||_{L^2(\Omega)}\leq \gamma_{\lambda} ||X_{N(\lambda)}||_2\leq \gamma_{\lambda} ||y(t)||_{L^2(\Omega)}\leq \lambda C_1^2 e^{2C_1 \sqrt{\lambda}} e^{-\frac{\lambda}{2} t} ||y(0)||_{L^2(\Omega)}.  \notag
\end{equation}

By applying the above explicit feedback law, we get the following  rapid stabilization result.  For any $\lambda>0$ we define an explicit stationary feedback law $\mathcal{F}_{\lambda}: L^2(\Omega)\rightarrow L^2(\Omega)$,
\begin{equation}\label{def-Flambda}
\mathcal{F}_{\lambda} y:= -\gamma_{\lambda}  \left( \sum_{i=1}^{N(\lambda)} \Big(y, e_i\Big)_{L^2(\Omega)} e_i \right)= -\gamma_{\lambda} P_{N(\lambda)} y \textrm{ with } \gamma_{\lambda}= C_1 e^{C_1\sqrt{\lambda}} \lambda,
\end{equation}
where  $P_{N(\lambda)}$ is the projection on the sub-space spanned by $\{e_i\}_{i=1}^{N(\lambda)}$, and $N(\lambda)$   is the number of eigenvalues (counting multiplicity) that are not strictly bigger than $\lambda$.  Clearly, there exists $C_2\geq 2C_1$ such that for all $\lambda>0$,
\begin{equation}\label{def-C2}
\lambda C_1^2 e^{2C_1 \sqrt{\lambda}},   \lambda C_1 e^{C_1 \sqrt{\lambda}},  C_1 e^{C_1 \sqrt{\lambda}}\leq C_2 e^{C_2\sqrt{\lambda}}=: \frac{1}{r_{\lambda}}.
\end{equation}

\begin{theorem}\label{thm-rap-sta-li}
For any $\lambda> 0$ the closed-loop  system 
\begin{gather*}
y_t= \Delta y+ 1_{\omega}  \mathcal{F}_{\lambda} y \; \textrm{ in } \Omega,\\
y=0 \; \textrm{ on } \partial \Omega,
\end{gather*} 
 is exponentially stable. More precisely,  for any $s\in \mathbb{R}$ the Cauchy problem 
\begin{gather*}
y_t= \Delta y-\gamma_{\lambda} 1_{\omega} \left( \sum_{i=1}^{N(\lambda)} \Big(y(t), e_i\Big)_{L^2(\Omega)} e_i \right), \; \forall (t, x)\in [s, +\infty)\times \Omega,\\
y(t, x)=0, \; \forall   (t, x)\in [s, +\infty)\times \partial \Omega, \\
y(s, x)= y_0(x),
\end{gather*} 
 has a unique solution in $C^0([s, +\infty); L^2(\Omega))\cap L^2_{loc}(s, +\infty; H^1_0(\Omega))$, and this unique solution verifies 
\begin{align}
||y(t)||_{L^2(\Omega)}&\leq C_1 e^{C_1\sqrt{\lambda}} e^{-\frac{\lambda}{2} (t-s)}||y_0||_{L^2(\Omega)}, \; \forall  \;t\in [s, +\infty), \label{ex-es-1} \\
||\mathcal{F}_{\lambda}y(t)||_{L^2(\Omega)}&\leq  C_2 e^{C_2\sqrt{\lambda}} e^{-\frac{\lambda}{2} (t-s)}||y_0||_{L^2(\Omega)}, \; \forall \; t\in [s, +\infty).    \label{ex-es-2}
\end{align}
\end{theorem}

\subsection{Rapid stabilization for higher regularity} 
Actually the feedback law $\mathcal{F}_{\lambda}$ presented in \eqref{def-Flambda} and Theorem \ref{thm-rap-sta-li}  also stabilizes the system in  $H^1_0(\Omega)$ space, let us briefly comment on this issue without going into  details.  We refer to Brezis \cite[Chapter 9--10]{Brezis-book} and Lions--Magenes \cite{1972-Lions-Magenes} for  related well-posedness results. 

Let the Hilbert space $H^1_0(\Omega)$ be endowed with  scalar product $\int_{\Omega} \nabla u\cdot \nabla v$.  The eigenfunctions $\{e_i/\sqrt{\tau_i}\}_{i=1}^{\infty}$ form an orthonormal basis of $H^1_0(\Omega)$. For any $y_0\in H^1_0(\Omega)$ and any $f(t, x)\in L^2(0, T; L^2(\Omega))$, the Cauchy problem  \eqref{Cauch-heat} admits a unique solution $y(t)$ in $C^0([0, T]; H^1_0(\Omega))\cap L^2(0, T; H^2(\Omega))$, thanks to  the \textit{a priori} estimate, 
\begin{equation*}
||\Delta y||_{L^2(0, t; L^2(\Omega))}^2+ ||\nabla y(t)||_{L^2(\Omega)}^2\leq ||\nabla y_0||_{L^2(\Omega)}^2+ ||f||_{L^2(0, t;L^2(\Omega))}^2, \; \forall \; t\in (0, T].
\end{equation*}
We further adapt the  notations in the preceding section, and  even the same choice of  $\gamma_{\lambda}$. For any  $\lambda>0$, we consider the Lyapunov functional on $H^1_0(\Omega)$: 
\[  V_1(y):= \widetilde{\mu}_{\lambda} ||X_{N(\lambda)}||_2^2+ ||P_{N(\lambda)}^T y||_{H^1_0(\Omega)}^2, \; \forall\; y\in H^1_0(\Omega), \]
with $\widetilde{\mu}_{\lambda}:= \gamma_{\lambda}^2/\lambda$.  It is actually equivalent to the $H^1_0(\Omega)$ norm by
\[  \frac{1}{1+ \lambda} ||y||_{H^1_0(\Omega)}^2\leq V_1(y)\leq \left(1+ \frac{\mu_{\lambda}}{\tau_1}\right)  ||y||_{H^1_0(\Omega)}^2, \; \forall\; y\in H^1_0(\Omega).  \]
Then, similar estimation implies that the solution $y(t)$ of the closed-loop system  \eqref{clo-1}--\eqref{cls-2} with feedback law $\mathcal{F}_{\lambda}$ verifies,  where $N$ means $N(\lambda)$,
\begin{align*}
\dot{V}_1(y(t))&\leq -2\widetilde{\mu}_{\lambda}\gamma_{\lambda}  C_1^{-1} e^{-C_1\sqrt{\lambda}}||X_N||_2^2- 2\left(P_{N}^T \Delta y, \Delta y-\gamma_{\lambda} 1_{\omega}(P_{N}y)\right)_{L^2(\Omega)},\\
&\leq -2 \lambda \widetilde{\mu}_{\lambda} ||X_N||_2^2- \lambda ||P_{N}^T  y||_{H^1_0(\Omega)}^2- ||P_{N}^T \Delta y||_{L^2(\Omega)}^2+  ||P_{N}^T \Delta y||_{L^2(\Omega)}^2+ \gamma_{\lambda}^2 ||P_{N}  y||_{L^2(\Omega)}^2, \\
&\leq  -2 \lambda \widetilde{\mu}_{\lambda} ||X_N||_2^2- \lambda ||P_{N}^T  y||_{H^1_0(\Omega)}^2+ \lambda \widetilde{\mu}_{\lambda} ||X_N||_2^2, \\
&\leq -\lambda V_1(y(t)).
\end{align*}
The stabilization on $H^1_0(\Omega)$ space becomes more important when it is combined with the Sobolev embedding $H^1(\Omega)\subseteq L^p(\Omega)$ with $p= \frac{2d}{d-2}$.  As for stabilization for even higher regularities, $H^2(\Omega)\cap H^1_0(\Omega)$ for example, probably one needs to replace the control setting $1_{\omega}f$ by $\chi_{\omega} f$  with some smooth truncated function  $\chi_{\omega}(x)$  that is supported in  $\omega$ and equals to 1 in an open subset $\omega_1\subset \omega$.

\section{Null controllability with optimal cost estimates}\label{sec-null}
Armed with  the   $Ce^{C\sqrt{\lambda}}$ estimates \eqref{ex-es-1}--\eqref{ex-es-2}, exactly the same procedure proposed  in  \cite{2017-Coron-Nguyen-ARMA, 2017-Xiang-SCL,  2019-xiang-SICON} by using piecewise stabilizing controls  leads to the null controllability.   In this section we construct similar feedback laws while keeping an extra attention on  control costs.  Two different kind of  precise feedback laws (control) are considered with control  costs  $Ce^{\frac{C}{T^{1+\varepsilon}}}$  and $Ce^{\frac{C}{T}}$ respectively.

 We mainly focus on the following weaker result, Theorem \ref{thm-null-col}, for which the feedback law (control) is nice and the calculation is easy. After that easy modification leads to  stronger cases, Corollary \ref{cor-null-col} and Theorem \ref{thm-null-col-opt}.
\begin{theorem}\label{thm-null-col}
There exists $C_3>0$ such that, for any $T\in (0, 1)$ and for any $y_0\in L^2(\Omega)$, we find an explicit control $f(t, x)$ for the control system \eqref{cons1}--\eqref{cons2} such that the unique solution with initial data $y(0, x)=y_0(x)$ verifies $y(T)=0$.  Moreover, the controlling cost is given by,
\begin{equation}
|| 1_{\omega} f(t, x)||_{L^{\infty}(0, T; L^2(\Omega))}\leq   e^{C_3/T^2} ||y_0||_{L^2(\Omega)}.  \notag
\end{equation} 
\end{theorem}
\begin{proof}[Proof of Theorem \ref{thm-null-col}]
We only treat the case  $1/T$ be integer to simplify the presentation.   Let us take some $\Gamma> 0$ independent of $T\in (0, 1)$ that will be fixed later on.  
\\

\noindent\textbf{Control design.}
\textit{Let $T=\frac{1}{n_T}$ with $n_T\in\mathbb{N}^*$. We define,}
\begin{gather}\label{def-In-t-l}
T_{n}:= T-\frac{1}{n}, \; I_n:= [T_n, T_{n+1}), \; \lambda_n:= \Gamma^2 n^4 \textrm{ for } \forall \; n\geq n_T:=\frac{1}{T},\\
\textit{for any $n\geq n_T$  we consider the control  (feedback law) as $\mathcal{F}_{\lambda_n}$ on interval $I_n$. }  \notag
\end{gather}
\textit{More precisely, first on $I_{n_T}$ we consider the closed-loop system \eqref{cons1}--\eqref{cons2}  with feedback law $\mathcal{F}_{\lambda_{n_T}}$ and $y(0, x)= y_0(x)$.  According to Theorem \ref{thm-rap-sta-li} this system has a unique solution $\widetilde{y}|_{\bar{I}_{n_T}}$. Next, we consider  the closed-loop system  with feedback law $\mathcal{F}_{\lambda_{n_T+1}}$ and $y(T_{n_T+1}, x):= \widetilde{y}(T_{n_T+1}, x)$ on $I_{n_T+1}$,  which, again, admit a unique solution $\widetilde{y}|_{\bar{I}_{n_T+1}}$. We continue this procedure on $\{I_n\}_{n=n_T}^{\infty}$ to find eventually a function $\widetilde{y}|_{[0, T)}\in C^0([0, T); L^2(\Omega))$ such that   $\widetilde{y}(T):= \lim_{t\rightarrow T^-} \widetilde{y}(t)=0$,  and that }
\begin{gather}
\widetilde{y}|_{[0, T]} \textit{ is the solution of the Cauchy problem } \eqref{Cauch-heat} \textit{ with control } 1_{\omega}f|_{I_n}= 1_{\omega}\mathcal{F}_{\lambda_{n}} \widetilde{y}|_{I_n},  \forall n\geq n_T.  \notag
\end{gather}
\\
We denote the above constructed solution by $y(t)$ which,    by Theorem \ref{thm-rap-sta-li}, verifies
\begin{align}
||y(t)||_{L^2(\Omega)}&\leq C_1 e^{C_1\Gamma n^2} e^{-\frac{\Gamma^2 n^4}{2} (t-T_n)} ||y(T_n)||_{L^2(\Omega)}, \; \forall t\in I_n, \; \forall n\geq n_T,\label{yt1} \\
||\mathcal{F}_{\lambda_n}y(t)||_{L^2(\Omega)}&\leq  C_2 e^{C_2\Gamma n^2} e^{-\frac{\Gamma^2 n^4}{2} (t-T_n)} ||y(T_n)||_{L^2(\Omega)}, \; \forall t\in I_n, \; \forall n\geq n_T. \label{ft1}
\end{align}
Therefore,  for $n\geq n_T+1$ the value of the solution on $T_n$ is controlled by,
\begin{equation}\label{Tn}
||y(T_n)||_{L^2(\Omega)}\leq \prod_{k=n_T}^{n-1} \left( C_1 e^{C_1\Gamma k^2} e^{-\frac{\Gamma^2 k^2}{4} } \right) ||y_0||_{L^2(\Omega)}. 
\end{equation}
Inspired by the preceding estimates, we choose the constant $\Gamma>0$ be such that 
\begin{equation}\label{def-Gam}
C_1 e^{C_1\Gamma n^2}, C_2 e^{C_2\Gamma n^2} \leq e^{\frac{\Gamma^2}{16}n^2}, \; \forall n\in \mathbb{N}^*.
\end{equation}
The above choice of  $\Gamma$, combined with \eqref{Tn}, lead to
\begin{equation}\label{fn-es}
||y(T_n)||_{L^2(\Omega)}\leq \left(\prod_{k=n_T}^{n-1}  e^{-\frac{3\Gamma^2}{16}k^2}\right) ||y_0||_{L^2(\Omega)}, \; \forall \; n\geq n_T+1.
\end{equation}
Essentially, it already implies that $y(T_n)$ is strictly decaying to 0 at time $T$. Next, we concentrate on its cost, $i.e.$ the norm of the control term.  From \eqref{yt1}, \eqref{ft1}, \eqref{def-Gam}, and \eqref{fn-es} we know that for $n\geq n_T+1$ and $t\in I_n$,
 \begin{align*}
 ||y(t)||_{L^2(\Omega)},  ||\mathcal{F}_{\lambda_n}y(t)||_{L^2(\Omega)}&\leq e^{\frac{\Gamma^2}{16}n^2}  \left(\prod_{k=n_T}^{n-1}  e^{-\frac{3\Gamma^2}{16}k^2}\right) ||y_0||_{L^2(\Omega)}, \\
 &\leq  \textrm{exp} \left(-\frac{\Gamma^2}{16}\Big(3\big(\sum_{k=n_T}^{n-1}k^2\big)-n^2\Big) \right) ||y_0||_{L^2(\Omega)}\leq ||y_0||_{L^2(\Omega)}.
 \end{align*}
 On the other hand, for $n=n_T$ we know that 
 \begin{equation}\label{es-123}
  ||y(t)||_{L^2(\Omega)},  ||\mathcal{F}_{\lambda}y(t)||_{L^2(\Omega)}\leq e^{\frac{\Gamma^2}{16}n_T^2}||y_0||_{L^2(\Omega)}=  e^{\frac{\Gamma^2}{16T^2}}||y_0||_{L^2(\Omega)}= e^{\frac{C_3}{T^2}}||y_0||_{L^2(\Omega)},
 \end{equation}
 where $C_3:= \frac{\Gamma^2}{16}$.
 
 In conclusion,  the constructed solution $y(t, x)$ with control $1_{\omega}f(t, x)$ satisfies
 \begin{gather*}
 ||y(t)||_{L^2(\Omega)} \textrm{ and }||1_{\omega}f(t, \cdot)||_{L^2(\Omega)}\longrightarrow 0^+, \textrm{ as } t\rightarrow T^-, \\
 ||y(t)||_{L^2(\Omega)} \textrm{ and }||1_{\omega}f(t, \cdot)||_{L^2(\Omega)}\leq e^{\frac{\Gamma^2}{16T^2}}||y_0||_{L^2(\Omega)}, \; \forall\; t\in [0, T],
 \end{gather*}
which completes the proof.
\end{proof}

Actually Theorem \ref{thm-null-col} can be easily improved to the following one via simple modification on the choice of $T_n$ and $\lambda_n$.  
\begin{cor}\label{cor-null-col}
For any $\varepsilon\in (0, 1)$, there exists $C_3^{\varepsilon}>0$ such that, for any $T\in (0, 1)$, for any $y_0\in L^2(\Omega)$, we can find an explicit control $f(t, x)$ for the control system \eqref{cons1}--\eqref{cons2} such that the unique solution with initial data $y(0, x)= y_0(x)$ verifies $y(T, x)=0$.  Moreover, the cost is controlled by 
\begin{equation}
|| 1_{\omega} f(t, x)||_{L^{\infty}(0, T; L^2(\Omega))}\leq   e^{C_3^{\varepsilon}/T^{1+\varepsilon}} ||y_0||_{L^2}.  \notag
\end{equation} 
\end{cor}
\begin{proof}[Idea of the proof]
Indeed, it suffices to take 
\begin{equation}\label{cons.t-lam}
T_{k, n}:= T-\frac{1}{n^k}, \; \lambda_{k, n}:= \Gamma^2_{\varepsilon} n^{2(k+1)} \textrm{ for } \forall n\geq n_T:=\frac{1}{T^{1/k}},
\end{equation}
for some $k\geq 1/\varepsilon$,  and to find some suitable $\Gamma_{\varepsilon}$.  We observe that the energy decay on interval $I_{k, n}$ is dominated by 
\begin{equation}
C_1 e^{C_1 \Gamma_{\varepsilon} n^{k+1}} e^{-c\Gamma_{\varepsilon}^2 n^{k+1}},  \notag
\end{equation}
which allows us to find some $\Gamma_{\varepsilon}$ satisfying \eqref{def-Gam} type estimates.
\end{proof}
However, $e^{C/T^{1+\varepsilon}}$ is the best estimate that we can achieve from partitions of type  \eqref{cons.t-lam}, which is slightly weaker than   the optimal cost \cite{Miller-2010}: $e^{C/T}$.  Eventually with another choice of partition we can also get the optimal cost  from  stabilization approach.
\begin{theorem}[Optimal  cost]\label{thm-null-col-opt}
There exists $C_3^0>0$ such that, for any $T\in (0, 1)$ and for any $y_0\in L^2(\Omega)$, we find an explicit control $f(t, x)$ for the control system \eqref{cons1}--\eqref{cons2}  such that the unique solution with initial data $y(0, x)=y_0(x)$ verifies $y(T, x)=0$.  Moreover, the controlling cost is given by
\begin{equation}
|| 1_{\omega} f(t, x)||_{L^{\infty}(0, T; L^2(\Omega))}\leq   e^{C_3^0/T} ||y_0||_{L^2}.  \notag
\end{equation} 
\end{theorem}
\begin{proof}
As illustrated above we adapt another type of construction to get this optimal result.  For the ease of presentation, we only consider the case $1/T= 2^{n_0}$ with $n_0\in N^*$. More precisely, we consider the following  partition as well as the piecewise controlling method explained in the proof of Theorem \ref{thm-null-col} (see Control design),
\begin{gather*}\label{def-In-t-l.new}
T_{n}^0:= 2^{-n_0}\left(1- \frac{1}{2^n}\right), \; I_n^0:= [T_n^0, T_{n+1}^0), \; \lambda_n^0:= Q^2 2^{2(n_0+n)} \textrm{ for } \forall \;n\geq 0,
\end{gather*}
where $Q>0$ is a given constant  satisfying
\begin{equation*}
C_1 e^{C_1 Q m}, C_2e^{C_2 Q m}\leq e^{\frac{Q^2}{16} m}, \; \forall \; m\geq 1. 
\end{equation*}
Suppose that $ y(t)$ is the unique solution satisfying the designed control, then for $n\geq 1$ we are able to estimate $y(T_n)$ by
\begin{align*}
||y(T_n)||_{L^2(\Omega)}&\leq \prod_{k=0}^{n-1} \left(C_1 e^{C_1\sqrt{\lambda_k^0}} e^{-\frac{\lambda_k^0}{2} 2^{-(n_0+k+1)} } \right) ||y_0||_{L^2(\Omega)}, \\
&\leq \prod_{k=0}^{n-1} \left(C_1 e^{C_1Q 2^{n_0+k}}  e^{-\frac{Q^2}{4} 2^{n_0+k}}\right) ||y_0||_{L^2(\Omega)},\\
&\leq \prod_{k=0}^{n-1} \left( e^{-\frac{Q^2}{8} 2^{n_0+k}}  \right) ||y_0||_{L^2(\Omega)}.
\end{align*}
For $n\geq 1$, the preceding estimate  further implies that the control term on  $t\in I_n^0$ satisfies,
\begin{equation*}
||1_{\omega}\mathcal{F}_{\lambda_n^0} y(t)||_{L^2(\Omega)}\leq C_2 e^{C_2 Q 2^{n_0+ n}}||y(T_n)||_{L^2(\Omega)}\leq e^{\frac{Q^2}{16} 2^{n_0+n}}\prod_{k=0}^{n-1} \left( e^{-\frac{Q^2}{8} 2^{n_0+k}}  \right) ||y_0||_{L^2(\Omega)}\leq ||y_0||_{L^2(\Omega)}.
\end{equation*}
Therefore, the $L^{\infty}(0, T; L^2(\Omega))$ norm of the control term $1_{\omega} f$ is dominated by  its   $L^{\infty}(T_0^0, T_1^0; L^2(\Omega))$ norm. As a consequence, we know that for any $t\in [0, T]$,
\begin{equation*}
||1_{\omega} f(t, x)||_{L^2(\Omega)}\leq C_2 e^{C_2 Q 2^{n_0}} ||y_0||_{L^2(\Omega)}\leq e^{\frac{C_3^0}{T}} ||y_0||_{L^2(\Omega)} \textrm{ with } C_3^0= \frac{Q^2}{16}.
\end{equation*}
\end{proof}

\begin{remark}
It is  noteworthy  that the $e^{C/T}$ type cost is optimal to many other systems, for example the Stokes system \cite{CS-Lebeau-2016}  where similar spectral estimates are proved.
\end{remark}

\section{Finite time stabilization}\label{sec-finite}
In this section, we construct $T$-periodic proper feedback laws that stabilize  system \eqref{cons1}--\eqref{cons2} in finite time: Theorem \ref{semi-stab}.
\subsection{Time-varying feedback laws and finite time stabilization}\label{sec-def-prop-feed}
We are interested in time-varying feedback laws, more precisely proper feedback laws.    The following definition of time-varying feedback laws that allows  the closed-loop system admit a unique solution  borrows directly from the paper \cite{Coron-Xiang-2018}.  

First, we recall the  closed-loop system associated to a time-varying feedback law $U$.
\begin{equation}\label{sys-heat-clo-def}
\begin{cases}
y_t= \Delta y+ 1_{\omega} U(t; y), \; \forall   \; (t, x)\in (s, +\infty) \times \Omega,\\
y=0 \; \textrm{ on } \partial \Omega.
\end{cases}
\end{equation} 
\begin{defi}
\label{def-sol-closed-loop}
Let $s_1\in \R$ and $s_2\in \R$ be given such that $s_1<s_2$. Let
\begin{equation*}
\begin{array}{ccc}
U: [s_1,s_2]\times L^2(\Omega) &\to & L^2(\Omega)
\\
(t; y)&\mapsto & U(t; y).
\end{array}
\end{equation*}
Let $t_1\in [s_1,s_2]$, $t_2\in (t_1, s_2]$, and
$y_0\in L^2(\Omega)$. A solution on $[t_1,t_2]$ to the Cauchy problem associated to 
the closed-loop system \eqref{sys-heat-clo-def} with initial data $y_0$ at time $t_1$ is some $y:  [t_1,t_2]\to L^2(\Omega)$ such that
\begin{gather*}
t\in (t_1,t_2)\mapsto f(t, x):= U(t; y(t))\in L^2(t_1, t_2; L^2(\Omega)),
\\
y \text{ is a solution (see Definition \ref{def-sol-open}) of \eqref{Cauch-heat} with initial data $y_0$ at time $t_1$ and the above $1_{\omega}f(t, x)$.}
\end{gather*}
\end{defi}
The so-called \emph{proper}  feedback laws is a time-varying feedback law  such that the closed-loop system always admit a unique solution.
\begin{defi}
\label{defproper}
Let $s_1\in \R$ and $s_2\in \R$ be given such that $s_1<s_2$. A proper feedback law on $[s_1,s_2]$ is an application \begin{equation*}
\begin{array}{ccc}
U: [s_1,s_2]\times L^2(\Omega)&\to & L^2(\Omega)
\\
(t; y)&\mapsto & U(t; y).
\end{array}
\end{equation*}
such that, for every $t_1\in [s_1,s_2]$, for every $t_2\in (t_1, s_2]$, and for every
$y_0\in L^2(\Omega)$,  there exists a unique solution on $[t_1,t_2]$ to the Cauchy problem associated to 
the closed-loop system \eqref{sys-heat-clo-def} with initial data $y_0$ at time $t_1$ according to  Definition~\ref{def-sol-closed-loop}.

A proper feedback law  is an application $U$
\begin{equation*}
\begin{array}{ccc}
U:\R\times L^2(\Omega) &\to & L^2(\Omega)
\\
(t; y)&\mapsto & U(t; y).
\end{array}
\end{equation*}
such that, for every $s_1\in \R$ and for every $s_2\in \R$ satisfying $s_1<s_2$,
the feedback law restricted to $[s_1,s_2]\times L^2(\Omega)$ is a proper feedback law on
$[s_1,s_2]$.
\end{defi}

\noindent For a \textbf{proper} feedback law,    one can define the \textbf{flow} $\Phi: \Delta \times L^2(\Omega) \to L^2(\Omega)$, with $\Delta:=\{(t,s);\, t> s\}$ associated to this feedback law:   $\Phi(t,s; y_0)$ is the value at time $t$ of the solution $y$
to the closed-loop system \eqref{sys-heat-clo-def} which is equal to $y_0$ at time $s$.\\

Finally we  state the exact definition of the finite time stabilization.
\begin{defi}[Finite time stabilization of the heat equation]\label{def-fi-sta}
Let $T> 0$.  A $T$-periodic proper feedback law $U$ stabilizes the heat equation in finite time, if the flow $\Phi$ of  the closed-loop system \eqref{sys-heat-clo-def} verifies, 
\begin{itemize}
\item[(i)] ($2T$ stabilization) $\Phi(2T+ t, t; y_0)=0, \;\;\forall t\in \mathbb{R},\; \forall \; y_0\in L^2(\Omega).$
\item[(ii)] (Uniform stability)
For every  $\delta> 0$, there exists $\eta> 0$ such that
\begin{equation*}
\left(||(y_0||_{L^2(\Omega)}\leq \eta\right) \Rightarrow \left(||\Phi(t, t'; y_0)||_{L^2(\Omega)}\leq \delta, \;\forall   \; t'\in \R,\; \forall  \;t\in ( t', +\infty) \right).
\end{equation*}
\end{itemize}
\end{defi}

\subsection{Finite time stabilization}
We only work on the case when $1/T$ is an integer, as the other cases can be trivially treated by  time transition. 

Different from null controllability we do not pay extra attention to the ``stabilizing cost" with respect to $T$.  Indeed,  we directly apply the feedback law  constructed in the preceding section as  combination of stationary feedback laws $\mathcal{F}_{\lambda_n}$ on interval $I_n$. Then,  $\mathcal{F}_{\lambda_n}$ can be regarded as ``$\lambda_n$ frequency" feedback, which is  sensible with respect to  the  states for large $n$. For example, for some given $y_0\in L^2(\Omega)$ we consider the Cauchy problem of the closed-loop system with $\mathcal{F}_{\lambda_n}$  and $y(T_n)= y_0$. Thanks to Theorem \ref{thm-rap-sta-li}, $||y(t)||$ is uniformly bounded by $C_1e^{C_1\sqrt{\lambda}}||y_0||$ on $I_n$. By letting $n$ tends to $\infty$, we are not allowed to get ``uniform stability",  as commented in Remark \ref{rmk-why-unif}. \\
 Therefore, we introduce some truncated operator on feedback laws, especially for high frequencies $\lambda$, to guarantee ``uniform stability". However, in this case the natural \textit{a priori} bound for the Cauchy problem that can be expected is $C_{\varepsilon}||y_0||+ \varepsilon, \forall \varepsilon>0$. As a result the cost can not be bounded by $C||y_0||$, that explains why we do not characterize the stabilizing cost in details with respect to $T$. However, thanks to  the precise construction of the feedback laws that will be presented in this section, for any given $T$ an effectively computable stabilizing cost depending on ``starting time" and ``initial state" can be obtained. \\
 
 Before stating the detailed stabilizing theorem, we first recall the following notations and facts:
 \begin{gather*} 
 \textrm{truncated operator $\mathcal{K}_r:  L^2(\Omega)\rightarrow L^2(\Omega)$ such that $||\mathcal{K}_r (y)||_{L^2(\Omega)}\leq \min\{1, ||y||_{L^2(\Omega)}\}$, defined in \eqref{def-FF},}\\
  \textrm{ stationary feedback law $\mathcal{F}_{\lambda}: L^2(\Omega)\rightarrow L^2(\Omega) $ with $||\mathcal{F}_{\lambda}||\leq C_2e^{C_2\sqrt{\lambda}}= (r_{\lambda})^{-1}$, given by \eqref{def-Flambda},} \\
 T_{n}:= T-\frac{1}{n}, \;\; I_n:= [T_n, T_{n+1}), \;\; \lambda_n:= \Gamma^2 n^4 \textrm{ for } \forall n\geq n_T:=\frac{1}{T}, \textrm{ defined in \eqref{def-In-t-l}.}
 \end{gather*}
 
 \noindent Next, we show that the feedback law, $\mathcal{K}_{r_{\lambda_n}}\left( \mathcal{F}_{\lambda_n} y\right)$, satisfies
 \begin{equation}
 ||\mathcal{K}_{r_{\lambda_n}}\left( \mathcal{F}_{\lambda_n} y\right)||_{L^2(\Omega)}\leq \min \{1, \sqrt{2||y||_{L^2(\Omega)}}\}.
 \end{equation}
 Indeed, by the choice of $\mathcal{K}_{r_{\lambda_n}}$ and $\mathcal{F}_{\lambda_n}$ we prove the preceding inequality by two steps. 
 
  If $||\mathcal{F}_{\lambda_n} y||_{L^2(\Omega)}\leq 2 r_{\lambda_n}$, then 
 \begin{equation*}
 ||\mathcal{K}_{r_{\lambda_n}}\left( \mathcal{F}_{\lambda_n} y\right)||_{L^2(\Omega)}\leq || \mathcal{F}_{\lambda_n} y||_{L^2(\Omega)}\leq \sqrt{2 r_{\lambda_n} || \mathcal{F}_{\lambda_n} y||_{L^2(\Omega)}}\leq \sqrt{2||y||_{L^2(\Omega)}},
 \end{equation*}
 moreover, it is also bounded by 1 as $2 r_{\lambda_n}\leq 1$.
 
  If $||\mathcal{F}_{\lambda_n} y||_{L^2(\Omega)}> 2 r_{\lambda_n}$, then by the choice of $\mathcal{K}_{r_{\lambda_n}}$ we know that  $\mathcal{K}_{r_{\lambda_n}}\left( \mathcal{F}_{\lambda_n} y\right)= 0$.

\begin{theorem}[Semi-global finite time stabilization of the heat equation]\label{semi-stab}
Let $T= 1/n_T\in (0, 1)$ with $n_T\in \mathbb{N}^*$.  Let $\Lambda\geq 1$.  For any integer $N_T>n_T$, the $T$-periodic feedback law $U(t; y): \R\times L^2(\Omega)\rightarrow L^2(\Omega)$ given by 
\begin{gather}\label{feed-li-stab}
U\big{|}_{[0, T)\times L^2(\Omega)}(t; y):= 
\begin{cases}
\mathcal{F}_{\lambda_n} y, \;\;\; \forall \; y\in L^2(\Omega),  \forall\; t\in I_n, \forall\; n_T\leq n\leq N_T, \\
\mathcal{K}_{r_{\lambda_n}}\left(\mathcal{F}_{\lambda_n} y\right), \; \forall \; y\in L^2(\Omega),  \forall\; t\in I_n, \forall\;  n\geq N_T+1,
\end{cases}
\end{gather}
is a proper feedback law for  system \eqref{sys-heat-clo-def}. 

 Moreover, for an effectively computable large $N_T$ the feedback law \eqref{feed-li-stab} stabilizes system \eqref{sys-heat-clo-def} in finite time:
\begin{itemize}
\item[(i)] ($2T$ stabilization) $\Phi(2T+ t, t; y_0)=0, \;\;\forall \;t\in \mathbb{R},\; \forall\; ||y_0||_{L^2(\Omega)}\leq \Lambda.$
\item[(ii)] (Uniform stability)
For every  $\delta> 0$, there exists an effectively computable $\eta> 0$ such that
\begin{equation*}
\left(||(y_0||_{L^2(\Omega)}\leq \eta\right) \Rightarrow \left(||\Phi(t, t'; y_0)||_{L^2(\Omega)}\leq \delta, \;\forall \;  t'\in \R,\; \forall\; t\in ( t', +\infty) \right).
\end{equation*}
\end{itemize}
\end{theorem}

\begin{proof}[Proof of Theorem \ref{feed-li-stab}]
Thanks to the $Ce^{C\sqrt{\lambda}}$ estimate, the proof of Theorem \ref{semi-stab} is rather standard. Here we mimic the treatment for similar results on one dimensional parabolic equations  \cite{2017-Coron-Nguyen-ARMA}.   The proof is followed by three steps: the feedback law is proper; condition $(i)$;  and condition $(ii)$.

\textit{Step 1.}
First, we show that the feedback law given by \eqref{feed-li-stab} is proper.  Without loss of generality, we only need to prove that  for  any $ s\in [0, T)$ and for any $y_0\in L^2(\Omega)$  the Cauchy problem   
\begin{equation}
\begin{cases}
y_t= \Delta y+  1_{\omega} U(t; y), \;  (t, x)\in (s, T)\times \Omega,\\
y(t, x)=0, \;  (t, x)\in (s, T)\times \partial \Omega,\\
y(s, x)= y_0(x),  \notag
\end{cases}
\end{equation} 
has  a unique solution $y$, and $ \lim_{t\rightarrow T^{-}} y(t) \in L^2(\Omega)$.  Actually, the existence of a unique solution on each interval $I_n$ follows directly from Theorem \ref{thm-clo-sta}  for $n\leq N_T$ and from Theorem \ref{thm-clo-sta-F} for $n>N_T$. Hence,
\[y|_{[s, T)}(t)\in C^0([s, T); L^2(\Omega)).\]
Moreover,  $||y(t)||_{L^2(\Omega)}$ is uniformly bounded on $[s, T)$ thanks to Theorem \ref{thm-rap-sta-li} and Theorem \ref{thm-inho-es}. Therefore, the control term on time interval $[s, T)$ is uniformly bounded in $L^2(\Omega)$, thus by applying Theorem  \ref{thm-inho-es} again we know   that
\[y|_{[s, T]}(t)\in C^0([s, T]; L^2(\Omega))\cap L^2(s,T; H^1_0(\Omega)).\]
Or equivalently, $\lim_{t\rightarrow T^{-}} y(t)\in L^2(\Omega)$ can be proved by the Cauchy sequence argument suggested in \cite[page 1018 for (4.42)]{2017-Coron-Nguyen-ARMA}.  Therefore, the flow $\Phi(s, t; y)$ is well-defined on $\Delta \times L^2(\Omega)$.

\textit{Step 2.}
Next, we need to find a suitable integer $N_T$ such that the proper feedback law \eqref{feed-li-stab} stabilize  system \eqref{sys-heat-clo-def} in finite time, mainly focus on  condition $(i)$.
\begin{lemma} \label{dire-bound}
The following energy estimate concerning the flow of the closed-loop system holds, 
\begin{equation}\label{es-lem-1}
||\Phi(T, s; y_0)||_{L^2(\Omega)}\leq 2+ e^{\frac{\Gamma^2}{16} N_T^2}  ||y_0||_{L^2(\Omega)}, \; \forall \; ||y_0||_{L^2(\Omega)}\leq \Lambda, \; \forall \; s\in [0, T).
\end{equation}
\end{lemma}
\begin{proof}[Proof of Lemma \ref{dire-bound}]
For any given $||y_0||_{L^2(\Omega)}\leq \Lambda$. 

If $s\in [T_{N_T+1}, T)$, then since the feedback is bounded by 1, Theorem \ref{thm-inho-es} yields 
\[ ||\Phi(T, s; y_0)||_{L^2(\Omega)}\leq ||y_0||_{L^2(\Omega)}+ 2.\]

If $s\in [0, T_{N_T+1})$, then we estimate $||\Phi(T_{N_T+1}, s; y_0)||_{L^2(\Omega)}$. Suppose that $s\in I_n$ with $n_T\leq n\leq N_T$, then direct calculation shows that (recalling some estimates from  Section \ref{sec-null}, especially \eqref{yt1}--\eqref{fn-es}),
\begin{align*}
||\Phi(T_{N_T+1}, s; y_0)||_{L^2(\Omega)}&\leq C_1 e^{C_1\sqrt{\lambda_{n}}} \prod_{k=n+1}^{N_T} \left(C_1 e^{C_1\sqrt{\lambda_k}} e^{-\frac{\lambda_k}{2} (T_{k+1}-T_k)} \right)  ||y_0||_{L^2(\Omega)}, \\
&\leq C_1 e^{C_1 \Gamma n^2} \prod_{k=n+1}^{N_T} \left(C_1 e^{C_1 \Gamma k^2} e^{-\frac{\Gamma^2}{4} k^2} \right) ||y_0||_{L^2(\Omega)},\\
&\leq e^{\frac{\Gamma^2}{16} n^2}  \prod_{k=n+1}^{N_T} \left(e^{-\frac{3\Gamma^2}{16} k^2} \right) ||y_0||_{L^2(\Omega)}, \\
&\leq  e^{\frac{\Gamma^2}{16} N_T^2} ||y_0||_{L^2(\Omega)}.
\end{align*}
Next, for $\tilde{y}(T_{N_T+1}):= \Phi(T_{N_T+1}, s; y_0)$ we adapt the case that $s\in [T_{N_T+1}, T)$ to get the required result.
\end{proof}
By applying  Lemma \ref{dire-bound}  we know that for any $s\in [0, T)$ and for any $||y_0||_{L^2(\Omega)}\leq \Lambda$, 
\[ ||\Phi(T, s; y_0)||_{L^2(\Omega)}\leq 2+ e^{\frac{\Gamma^2}{16} N_T^2} \Lambda.  \]
Let us  define $\widetilde{y}(T):= \Phi(T, s; y_0)$.
The next step is to show that $\Phi(2T, T; \widetilde{y}(T))=0$, which requires us to 
seek for suitable $N_T$ such that for every $n\geq N_T+1$ we have, 
\begin{equation}\label{str-es-1}
K_{r_{\lambda_n}}\left(\mathcal{F}_{\lambda_n} \Phi(t,T;\widetilde{y}(T))  \right)= \mathcal{F}_{\lambda_n} \Phi(t,T;\widetilde{y}(T)),  \forall\; t\in I_n+T.
\end{equation}
For ease of notations we simply denote the unique solution of the closed-loop system  by 
\[\widetilde{y}(t):= \Phi(t,T;\widetilde{y}(T)),  \forall  \;  t\in [T, 2T]. \]
Thus, in order to prove \eqref{str-es-1} it suffices to show that 
\begin{equation}
||\mathcal{F}_{\lambda_n}\widetilde{y}(t)||_{L^2(\Omega)}\leq r_{\lambda_n}, \forall\;  t\in I_n+T, \; \forall \; n\geq N_T+1.  \notag
\end{equation}
As we know from the proof of Lemma \ref{dire-bound}, or from \eqref{fn-es},  that 
\begin{equation}
||\widetilde{y}(T_{N_T+1})||_{L^2(\Omega)}\leq   \prod_{k=n_T}^{N_T} \left(e^{-\frac{3\Gamma^2}{16} k^2} \right) ||\widetilde{y}(T)||_{L^2(\Omega)}\leq 2e^{\frac{\Gamma^2}{16} N_T^2} \Lambda \prod_{k=n_T}^{N_T} \left(e^{-\frac{3\Gamma^2}{16} k^2} \right),   \notag
\end{equation}
and that for $n\geq N_T+1$,
\begin{equation}
||\widetilde{y}(T_{n})||_{L^2(\Omega)}\leq 2e^{\frac{\Gamma^2}{16} N_T^2} \Lambda \prod_{k=n_T}^{N_T} \left(e^{-\frac{3\Gamma^2}{16} k^2} \right) \left( \prod_{k=N_T+1}^{n-1} \left(e^{-\frac{3\Gamma^2}{16} k^2} \right) \right),   \notag
\end{equation}
and that for $n\geq N_T+1$ and $t\in I_n+T$,
\begin{equation}
||\mathcal{F}_{\lambda_n}\widetilde{y}(t)||_{L^2(\Omega)}\leq ||\widetilde{y}(T_{n})||_{L^2(\Omega)} C_2 e^{C_2\Gamma n^2} \leq ||\widetilde{y}(T_{n})||_{L^2(\Omega)} e^{\frac{\Gamma^2}{16}n^2}.      \notag
\end{equation}
Therefore,  it suffices to find $N_T$ such that for every $n\geq N_T+1$ we have 
\begin{equation}
\left(2e^{\frac{\Gamma^2}{16} N_T^2} \Lambda \prod_{k=n_T}^{N_T} \left(e^{-\frac{3\Gamma^2}{16} k^2} \right) \right)\left( \prod_{k=N_T+1}^{n-1} \left(e^{-\frac{3\Gamma^2}{16} k^2} \right) \right)e^{\frac{\Gamma^2}{16}n^2}\leq e^{-\frac{\Gamma^2}{16}n^2}.   \notag
\end{equation}
Thus one only needs to find the existence of $N_T$ such that 
\begin{equation}
2e^{\frac{\Gamma^2}{16} N_T^2} \Lambda \left( \prod_{k=n_T}^{N_T} \left(e^{-\frac{3\Gamma^2}{16} k^2} \right) \right)  e^{\frac{\Gamma^2}{8}(N_T+1)^2}\leq 1,   \notag
\end{equation} 
  which is obviously possible for any given $\Lambda>1$.
  
  \textit{Step 3.}  Finally, in order to complete the proof of finite time stabilization,  it only remains to  prove that the proper feedback law given by \eqref{feed-li-stab} satisfies condition $(ii)$: uniform stability.
  
Thanks to Step 2 we know the existence of $C$ such that 
\begin{equation}
||\Phi(t, T; y_0)||_{L^2(\Omega)}\leq C ||y_0||_{L^2(\Omega)}, \forall \;||y_0||_{L^2(\Omega)}\leq 1, \forall \;t\in [T, +\infty). \notag
\end{equation}
As a consequence for any $\delta>0$ there exists $\widetilde{\eta}\in (0, \delta)$ such that 
\begin{equation}\label{es-1}
||\Phi(t, T; y_0)||_{L^2(\Omega)}\leq  \delta, \;\forall\; ||y_0||_{L^2(\Omega)}\leq \widetilde{\eta},\; \forall \;t\in [T, +\infty).
\end{equation}

Moreover, there exists some $\varepsilon>0$ such that 
\begin{equation}\label{es-2}
||\Phi(t, s; y_0)||_{L^2(\Omega)}\leq  \widetilde{\eta}, \;\forall \;||y_0||_{L^2(\Omega)}\leq \varepsilon, \;\forall \;s\in [0, T), \;\forall \;t\in [s, T].
\end{equation}
Indeed, thanks to Theorem \ref{thm-inho-es} and the fact that $||\mathcal{K}_{r_{\lambda_n}}(y)||\leq 1$, there exists $\tilde{T}\in (0, T)$ such that 
\begin{equation}\label{es-3}
||\Phi(t, s; y_0)||_{L^2(\Omega)}\leq  \widetilde{\eta}, \;\forall \;||y_0||_{L^2(\Omega)}\leq \widetilde{\eta}/2, \;\forall \;s\in [\tilde{T}, T), \;\forall \;t\in [s, T].
\end{equation}
Because the time-varying feedback law $U$ on $[0, \tilde{T})$ is given by finitely many stationary feedback laws, there exists some $\varepsilon\in (0, \widetilde{\eta}/2)$ such that 
\begin{equation}\label{es-4}
||\Phi(t, s; y_0)||_{L^2(\Omega)}\leq  \widetilde{\eta}/2,\; \forall \;||y_0||_{L^2(\Omega)}\leq \varepsilon, \;\forall \;s\in [0, \tilde{T}), \;\forall \;t\in [s, \tilde{T}].
\end{equation}
In conclusion, inequalities \eqref{es-3}--\eqref{es-4} yields \eqref{es-2}; then estimates \eqref{es-1}--\eqref{es-2}, as well as the fact that $\Phi(2T, s; y_0)=0$, imply the uniform stability condition $(ii)$.
 \end{proof}
 

\bibliographystyle{plain} 
\bibliography{rapheat}

\begin{thebibliography}{10}

\bibitem{Barbu-MAMS}
Viorel Barbu, Irena Lasiecka, and Roberto Triggiani.
\newblock Tangential boundary stabilization of {N}avier-{S}tokes equations.
\newblock {\em Mem. Amer. Math. Soc.}, 181(852):x+128, 2006.

\bibitem{Barbu-Triggiani}
Viorel Barbu and Roberto Triggiani.
\newblock Internal stabilization of {N}avier-{S}tokes equations with
  finite-dimensional controllers.
\newblock {\em Indiana Univ. Math. J.}, 53(5):1443--1494, 2004.

\bibitem{Bardo-Lebeau-Rauch}
Claude Bardos, Gilles Lebeau, and Jeffrey Rauch.
\newblock Sharp sufficient conditions for the observation, control, and
  stabilization of waves from the boundary.
\newblock {\em SIAM J. Control Optim.}, 30(5):1024--1065, 1992.

\bibitem{Kunisch-2014}
Tobias Breiten and Karl Kunisch.
\newblock Riccati-based feedback control of the monodomain equations with the
  {F}itz{H}ugh-{N}agumo model.
\newblock {\em SIAM J. Control Optim.}, 52(6):4057--4081, 2014.

\bibitem{Brezis-book}
Haim Brezis.
\newblock {\em Functional analysis, {S}obolev spaces and partial differential
  equations}.
\newblock Universitext. Springer, New York, 2011.

\bibitem{Burq-Gerard-wave}
Nicolas Burq and Patrick G\'{e}rard.
\newblock Condition n\'{e}cessaire et suffisante pour la contr\^{o}labilit\'{e}
  exacte des ondes.
\newblock {\em C. R. Acad. Sci. Paris S\'{e}r. I Math.}, 325(7):749--752, 1997.

\bibitem{Carleman}
T.~Carleman.
\newblock Sur un probl\`eme d'unicit\'{e} pur les syst\`emes d'\'{e}quations
  aux d\'{e}riv\'{e}es partielles \`a deux variables ind\'{e}pendantes.
\newblock {\em Ark. Mat., Astr. Fys.}, 26(17):9, 1939.

\bibitem{CS-Lebeau-2016}
Felipe~W. Chaves-Silva and Gilles Lebeau.
\newblock Spectral inequality and optimal cost of controllability for the
  {S}tokes system.
\newblock {\em ESAIM Control Optim. Calc. Var.}, 22(4):1137--1162, 2016.

\bibitem{coron}
Jean-Michel Coron.
\newblock {\em Control and nonlinearity}, volume 136 of {\em Mathematical
  Surveys and Monographs}.
\newblock American Mathematical Society, Providence, RI, 2007.

\bibitem{coron-2015}
Jean-Michel Coron.
\newblock Stabilization of control systems and nonlinearities.
\newblock In {\em Proceedings of the 8th {I}nternational {C}ongress on
  {I}ndustrial and {A}pplied {M}athematics}, pages 17--40. Higher Ed. Press,
  Beijing, 2015.

\bibitem{Coron-Gagnon}
Jean-Michel Coron, Ludovick Gagnon, and Morgan Morancey.
\newblock Rapid stabilization of a linearized bilinear 1-{D} {S}chr\"{o}dinger
  equation.
\newblock {\em J. Math. Pures Appl. (9)}, 115:24--73, 2018.

\bibitem{coronluqi}
Jean-Michel Coron and Qi~L{\"u}.
\newblock Local rapid stabilization for a {K}orteweg-de {V}ries equation with a
  {N}eumann boundary control on the right.
\newblock {\em J. Math. Pures Appl. (9)}, 102(6):1080--1120, 2014.

\bibitem{2017-Coron-Nguyen-ARMA}
Jean-Michel Coron and Hoai-Minh Nguyen.
\newblock Null controllability and finite time stabilization for the heat
  equations with variable coefficients in space in one dimension via
  backstepping approach.
\newblock {\em Arch. Ration. Mech. Anal.}, 225(3):993--1023, 2017.

\bibitem{Coron-trelat-2004}
Jean-Michel Coron and Emmanuel Tr\'{e}lat.
\newblock Global steady-state controllability of one-dimensional semilinear
  heat equations.
\newblock {\em SIAM J. Control Optim.}, 43(2):549--569, 2004.

\bibitem{Coron-Trelat-2006}
Jean-Michel Coron and Emmanuel Tr\'{e}lat.
\newblock Global steady-state stabilization and controllability of 1{D}
  semilinear wave equations.
\newblock {\em Commun. Contemp. Math.}, 8(4):535--567, 2006.

\bibitem{Coron-Xiang-2018}
Jean-Michel Coron and Shengquan Xiang.
\newblock Small-time global stabilization of the viscous {B}urgers equation
  with three scalar controls.
\newblock {\em Preprint, hal-01723188}, 2018.

\bibitem{D-Fefferman}
Harold Donnelly and Charles Fefferman.
\newblock Nodal sets of eigenfunctions on {R}iemannian manifolds.
\newblock {\em Invent. Math.}, 93(1):161--183, 1988.

\bibitem{Sylvain-Zuazua-2011}
Sylvain Ervedoza and Enrique Zuazua.
\newblock Sharp observability estimates for heat equations.
\newblock {\em Arch. Ration. Mech. Anal.}, 202(3):975--1017, 2011.

\bibitem{Fatto-Russell}
H.~O. Fattorini and D.~L. Russell.
\newblock Exact controllability theorems for linear parabolic equations in one
  space dimension.
\newblock {\em Arch. Rational Mech. Anal.}, 43:272--292, 1971.

\bibitem{FGIP-2004}
E.~Fern\'{a}ndez-Cara, S.~Guerrero, O.~Yu. Imanuvilov, and J.-P. Puel.
\newblock Local exact controllability of the {N}avier-{S}tokes system.
\newblock {\em J. Math. Pures Appl. (9)}, 83(12):1501--1542, 2004.

\bibitem{F-Zuazua-2}
Enrique Fern\'{a}ndez-Cara and Enrique Zuazua.
\newblock The cost of approximate controllability for heat equations: the
  linear case.
\newblock {\em Adv. Differential Equations}, 5(4-6):465--514, 2000.

\bibitem{F-Zuazua-1}
Enrique Fern\'{a}ndez-Cara and Enrique Zuazua.
\newblock Null and approximate controllability for weakly blowing up semilinear
  heat equations.
\newblock {\em Ann. Inst. H. Poincar\'{e} Anal. Non Lin\'{e}aire},
  17(5):583--616, 2000.

\bibitem{Fursikov-Imanuvilov-book-1997}
Andrei~V. Fursikov and Oleg~Yu. Imanuvilov.
\newblock {\em Controllability of evolution equations}, volume~34 of {\em
  Lecture Notes Series}.
\newblock Seoul National University, Research Institute of Mathematics, Global
  Analysis Research Center, Seoul, 1996.

\bibitem{Hormander-firstbook}
Lars H\"{o}rmander.
\newblock {\em Linear partial differential operators}.
\newblock Die Grundlehren der mathematischen Wissenschaften, Bd. 116. Academic
  Press, Inc., Publishers, New York; Springer-Verlag,
  Berlin-G\"{o}ttingen-Heidelberg, 1963.

\bibitem{Hormander-97}
Lars H\"{o}rmander.
\newblock On the uniqueness of the {C}auchy problem under partial analyticity
  assumptions.
\newblock In {\em Geometrical optics and related topics ({C}ortona, 1996)},
  volume~32 of {\em Progr. Nonlinear Differential Equations Appl.}, pages
  179--219. Birkh\"{a}user Boston, Boston, MA, 1997.

\bibitem{2008-Krstic-Smyshlyaev-book}
Miroslav Krstic and Andrey Smyshlyaev.
\newblock {\em Boundary control of {PDE}s}, volume~16 of {\em Advances in
  Design and Control}.
\newblock Society for Industrial and Applied Mathematics (SIAM), Philadelphia,
  PA, 2008.
\newblock A course on backstepping designs.

\bibitem{L-L-JEMS}
Camille Laurent and Matthieu L\'{e}autaud.
\newblock Quantitative unique continuation for operators with partially
  analytic coefficients. {A}pplication to approximate control for waves.
\newblock {\em J. Eur. Math. Soc. (JEMS)}, 21(4):957--1069, 2019.

\bibitem{L-L-MAMS}
Camille Laurent and Matthieu Léautaud.
\newblock Tunneling estimates and approximate controllability for hypoelliptic
  equations.
\newblock {\em Mem. A.M.S.}, to appear.

\bibitem{LR-Lebeau}
J\'{e}r\^{o}me Le~Rousseau and Gilles Lebeau.
\newblock On {C}arleman estimates for elliptic and parabolic operators.
  {A}pplications to unique continuation and control of parabolic equations.
\newblock {\em ESAIM Control Optim. Calc. Var.}, 18(3):712--747, 2012.

\bibitem{Lebeau-Robbiano-CPDE}
Gilles Lebeau and Luc Robbiano.
\newblock Contr\^ole exact de l'\'equation de la chaleur.
\newblock {\em Comm. Partial Differential Equations}, 20(1-2):335--356, 1995.

\bibitem{Lions-opt}
J.-L. Lions.
\newblock {\em Optimal control of systems governed by partial differential
  equations}.
\newblock Translated from the French by S. K. Mitter. Die Grundlehren der
  mathematischen Wissenschaften, Band 170. Springer-Verlag, New York-Berlin,
  1971.

\bibitem{Lions}
J.-L. Lions.
\newblock {\em Contr\^{o}labilit\'{e} exacte, perturbations et stabilisation de
  syst\`emes distribu\'{e}s. {T}ome 2}, volume~9 of {\em Recherches en
  Math\'{e}matiques Appliqu\'{e}es [Research in Applied Mathematics]}.
\newblock Masson, Paris, 1988.
\newblock Perturbations. [Perturbations].

\bibitem{1972-Lions-Magenes}
Jacques-Louis Lions and Enrico Magenes.
\newblock {\em Non-homogeneous boundary value problems and applications. {V}ol.
  {III}}.
\newblock Springer-Verlag, New York-Heidelberg, 1973.
\newblock Translated from the French by P. Kenneth, Die Grundlehren der
  mathematischen Wissenschaften, Band 183.

\bibitem{Miller-2010}
Luc Miller.
\newblock A direct {L}ebeau-{R}obbiano strategy for the observability of
  heat-like semigroups.
\newblock {\em Discrete Contin. Dyn. Syst. Ser. B}, 14(4):1465--1485, 2010.

\bibitem{Raymond-NS}
J.-P. Raymond.
\newblock Feedback boundary stabilization of the three-dimensional
  incompressible {N}avier-{S}tokes equations.
\newblock {\em J. Math. Pures Appl. (9)}, 87(6):627--669, 2007.

\bibitem{Robbiano-Zuily}
Luc Robbiano and Claude Zuily.
\newblock Uniqueness in the {C}auchy problem for operators with partially
  holomorphic coefficients.
\newblock {\em Invent. Math.}, 131(3):493--539, 1998.

\bibitem{Russell-1978}
David~L. Russell.
\newblock Controllability and stabilizability theory for linear partial
  differential equations: recent progress and open questions.
\newblock {\em SIAM Rev.}, 20(4):639--739, 1978.

\bibitem{Tataru-95}
Daniel Tataru.
\newblock Unique continuation for solutions to {PDE}'s; between
  {H}\"{o}rmander's theorem and {H}olmgren's theorem.
\newblock {\em Comm. Partial Differential Equations}, 20(5-6):855--884, 1995.

\bibitem{Trelat-stab-book}
Emmanuel Tr\'{e}lat.
\newblock Stabilization of semilinear {PDE}s, and uniform decay under
  discretization.
\newblock In {\em Evolution equations: long time behavior and control}, volume
  439 of {\em London Math. Soc. Lecture Note Ser.}, pages 31--76. Cambridge
  Univ. Press, Cambridge, 2018.

\bibitem{2017-Xiang-SCL}
Shengquan Xiang.
\newblock Small-time local stabilization for a {K}orteweg-de {V}ries equation.
\newblock {\em Systems \& Control Letters}, 111:64 -- 69, 2018.

\bibitem{2019-xiang-SICON}
Shengquan Xiang.
\newblock Null controllability of a linearized {K}orteweg-de {V}ries equation
  by backstepping approach.
\newblock {\em SIAM J. Control Optim.}, 57(2):1493--1515, 2019.

\bibitem{ZhangRapidStab}
Christophe Zhang.
\newblock {Internal rapid stabilization of a 1-D linear transport equation with
  a scalar feedback}.
\newblock Preprint, October 2018.

\bibitem{Zhang-finite-2019}
Christophe Zhang.
\newblock Finite-time internal stabilization of a linear 1-{D} transport
  equation.
\newblock {\em Systems Control Lett.}, 133:104529, 8, 2019.

\end{thebibliography}

\end{document}